\documentclass[11pt,a4paper]{article}
\usepackage{amsthm,amssymb,amsfonts,amsmath,lineno,mathtools,graphicx,endnotes,etoolbox,color,epsfig,xspace,algorithmic,algorithm,verbatim,tikz}
\usetikzlibrary{calc}

\def\caterpillar1{
\begin{tikzpicture}[scale=1.075,shorten >=1pt,->] 
\tikzstyle{vertex}=[circle,draw,minimum size=20pt,inner sep=0pt]
\foreach \name/\x in {0/0, 8/8} 
\node[vertex,fill=white] (\name) at (1.6*\x,0) {$u_\name$};
\foreach \name/\x in {1/1, 2/2, 3/3, 4/4, 5/5, 7/7} 
\node[vertex,fill=black!25] (\name) at (1.6*\x,0) {$u_\name$};
\foreach \name/\x in {6/6} 
\node[vertex,fill=black!5] (\name) at (1.6*\x,0) {$u_\name$};
\node[vertex,fill=white] (f1) at (0.6*1,-1.8) {};
\node[vertex,fill=white] (f2) at (1.6*1,-1.8) {};
\node[vertex,fill=white] (f3) at (2.6*1,-1.8) {};
\node[vertex,fill=white] (f4) at (1.45*4,-1.8) {};
\node[vertex,fill=white] (f4b) at (1.75*4,-1.8) {};
\node[vertex,fill=white] (f5) at (1.6*6,-1.8) {};
\node[vertex,fill=white] (f6) at (1.51*7,-1.8) {};
\node[vertex,fill=white] (f7) at (1.69*7,-1.8) {};
\draw (f1) -- node[left=5pt]{\scriptsize 7} (1);
\draw (f2) -- node[left=1pt]{\scriptsize 9} (1);
\draw (f3) -- node[right=2pt]{\scriptsize 11} (1);
\draw (4) -- node[left=2pt]{\scriptsize 5} (f4);
\draw (4) -- node[right=2pt]{\scriptsize 6} (f4b);
\draw (f5) -- node[left=3pt]{\scriptsize 12} (6);
\draw (7) -- node[left=2pt]{\scriptsize 10} (f6);
\draw (7) -- node[right=2pt]{\scriptsize 8} (f7);
\draw (0) -- node[above=5pt]{\scriptsize 4} (1);
\draw (2) -- node[above=5pt]{\scriptsize 16} (1);
\draw (2) -- node[above=5pt]{\scriptsize 3} (3);
\draw (4) -- node[above=5pt]{\scriptsize15} (3);
\draw (4) -- node[above=5pt]{\scriptsize 2} (5);
\draw (6) -- node[above=5pt]{\scriptsize 14} (5);
\draw (7) -- node[above=5pt]{\scriptsize 1} (6);
\draw (7) -- node[above=5pt]{\scriptsize 13} (8);
\draw (0) node[above=10pt]{\scriptsize -4};
\draw (1) node[above=10pt]{\scriptsize 47};
\draw (2) node[above=10pt]{\scriptsize -19};
\draw (3) node[above=10pt]{\scriptsize 18};
\draw (4) node[above=10pt]{\scriptsize -28};
\draw (5) node[above=10pt]{\scriptsize 16};
\draw (6) node[above=10pt]{\scriptsize -1};
\draw (7) node[above=10pt]{\scriptsize -32};
\draw (8) node[above=10pt]{\scriptsize 13};
\draw (f1) node[below=10pt]{\scriptsize -7};
\draw (f2) node[below=10pt]{\scriptsize -9};
\draw (f3) node[below=10pt]{\scriptsize -11};
\draw (f4) node[below=10pt]{\scriptsize 5};
\draw (f4b) node[below=10pt]{\scriptsize 6};
\draw (f5) node[below=10pt]{\scriptsize -12};
\draw (f6) node[below=10pt]{\scriptsize 10};
\draw (f7) node[below=10pt]{\scriptsize 8};
\end{tikzpicture}}

\def\path{
\begin{tikzpicture}[scale=1.09,shorten >=1pt,-] 
\tikzstyle{vertex}=[circle,draw,minimum size=24pt,inner sep=0pt]
\foreach \name/\x in {0/0,1/1, 2/2, 3/3, 4/4} 
\node[vertex,fill=black!0] (\name) at (1.7*\x,0) {\small $u_\name$};
\node[vertex,fill=black!0] (5) at (1.7*5,0) {\small $u_{k-2}$};
\node[vertex,fill=black!0] (6) at (1.7*6,0) {\small $u_{k-1}$};
\node[vertex,fill=black!0] (7) at (1.7*7,0) {\small $u_{k}$};
\draw (0) -- node[above=5pt]{\small $k_1$} (1);
\draw (1) -- node[above=5pt]{\small $m$} (2);
\draw (2) -- node[above=5pt]{\small $k_1-1$} (3);
\draw (3) -- node[above=5pt]{\small $m-1$} (4);
\draw[dotted] (4) -- (5);
\draw (5) -- node[above=5pt]{\small 1} (6);
\draw (6) -- node[above=5pt]{\small $k_2+1$} (7);

\end{tikzpicture}}

\usepackage{color}

\patchcmd{\theendnotes}
  {\makeatletter}
  {\makeatletter}
  {}{}

\DeclarePairedDelimiter{\ceil}{\lceil}{\rceil}
\DeclarePairedDelimiter{\floor}{\lfloor}{\rfloor}

\newtheorem{theorem}{Theorem}
\newtheorem{claim}{Claim}

\newtheorem{definition}{Definition}

\newtheorem{conjecture}{Conjecture}

\begin{document}

\title{Caterpillars Have Antimagic Orientations}

%\thanks{Supported by projects MTM2009-07242, Gen. Cat. DGR2009GR1040, TIN2007-68005-C04-03 and TIN2008-06582-C03-01.}

\author{Antoni Lozano\thanks{Computer Science Department,
Universitat Polit\`ecnica de Catalunya, Catalonia, {\tt antoni@cs.upc.edu}. Supported by  project MTM2014-600127-P.}}

\maketitle

%\linenumbers

\begin{abstract}
\noindent
An {\em antimagic labeling} of a directed graph $D$ with $m$ arcs is a bijection from the set of arcs of $D$ to $\{1,\dots,m\}$ such that all oriented vertex sums of vertices in $D$ are pairwise distinct, where the {\em oriented vertex sum} of a vertex $u$ is the sum of labels of all arcs entering $u$ minus the sum of labels of all arcs leaving $u$. Hefetz, M\"utze, and Schwartz~\cite{HMS} conjectured that every connected graph admits an antimagic orientation, where an {\em antimagic orientation} of a graph $G$ is an orientation of $G$ which has an antimagic labeling. We use a constructive technique to prove that caterpillars, a well-known subclass of trees, have antimagic orientations.
\end{abstract}

\section{Introduction}\label{sec:int}

All graphs considered in this paper are finite and simple unless otherwise stated. A {\em labeling} of a graph $G$ with $m$ edges is defined as a bijection from the set of edges of $G$ to the set $\{1,\dots,m\}$. A labeling of $G$ is said to be {\em antimagic} if all vertex sums are pairwise distinct, where the {\em vertex sum} of a vertex $u$ in $G$ is the sum of labels of all edges incident with $u$. A graph is said to be {\em antimagic} if it admits an antimagic labeling. Hartsfield and Ringel conjectured in~\cite{HR} that all simple connected graphs, with the exception of $K_2$, are antimagic. Although antimagicness has been proved for graphs belonging to many classes, like regular graphs~\cite{BBV,CLPZ} or trees having more than three vertices and at most one vertex of degree two~\cite{LWZ}, the conjecture is still open even for bipartite graphs. For related results the reader is referred to the survey of Gallian~\cite{G}.

Hartsfield and Ringel's conjecture finds a natural variation in the setting of directed graphs. A {\em labeling} of a directed graph $D$ with $m$ arcs is a bijection from the set of arcs of $D$ to the set $\{1,\dots,m\}$. A labeling of $D$ is said to be {\em antimagic} if all oriented vertex sums are pairwise distinct, where the {\em oriented vertex sum} of a vertex $u$ in $D$ is the sum of labels of all arcs entering $u$ minus the sum of labels of all arcs leaving $u$. A graph is said to have an {\em antimagic orientation} if it has an orientation which admits an antimagic labeling. Hefetz, M\"utze, and Schwartz~\cite{HMS} formulate the following conjecture:

\begin{conjecture}\label{con:hef}
Every connected graph admits an antimagic orientation.
\end{conjecture}

In the same article~\cite{HMS}, Hefetz {\em et al.} prove Conjecture~\ref{con:hef} for stars, wheels,  cliques, and ``dense" graphs  (graphs of order $n$ with minimum degree at least $C \log n$, for an absolute constant $C$); in fact, the authors prove the stronger statement that every orientation is antimagic. Other classes for which Conjecture~\ref{con:hef} is known to hold are odd regular graphs~\cite{HMS}, even regular graphs~\cite{LSWYZ}, or biregular bipartite graphs~\cite{SY}. Actually, it is easy to see that all antimagic bipartite graphs admit an antimagic orientation where all edges are oriented in the same direction between the partite sets. For this reason, all subclasses of trees that are known to be antimagic admit antimagic orientations. However, a particular subclass for which this reasoning cannot be applied is that of caterpillars.

\begin{definition}
A {\em caterpillar} $C$ is a tree of order at least 3 the removal of whose leaves produces a path. 
\end{definition}

When some specific conditions on the number of leaves or on the vertex degrees are added, caterpillars are known to be antimagic~\cite{LMS,K} but, in the general case, antimagicness is still open for caterpillars. Here we use the flexibility given by the choice of an orientation to adapt the constructive technique from Lozano, Mora, and Seara~\cite{LMS} and prove that all caterpillars admit an antimagic orientation, supporting Conjecture~\ref{con:hef}.
 
\section{Main Result}

Let $G$ be a graph and let $u$ and $v$ be vertices of $G$. Then, $E(G)$ denotes the set of edges of $G$ and $\{u,v\}$ represents an edge between $u$ and $v$. An orientation of $G$ will be usually represented by $\vec{G}$, the set of arcs of $\vec{G}$ by $A(\vec{G})$ and an arc from $u$ to $v$ by $uv$. For any labeling $\phi$ of a graph $G$ and any orientation $\vec{G}$ of $G$, we use the notation $\phi(uv)$, for an arc $uv \in A(\vec{G})$, to mean $\phi(\{u,v\})$. This way, $\phi$ is used to define both the {\em vertex sum} $s(u)$ of $u$ in $G$ and the {\em oriented vertex sum} $\vec{s}(u)$ of $u$ in $\vec{G}$:
\[ s(u) = \sum_{\{u,v\} \in E(G)} \phi(\{u,v\}), \;\; \mbox{ and }\;\; \vec{s}(u) = \sum_{vu \in A(\vec{G})} \phi(vu) - \sum_{uv \in A(\vec{G})} \phi(uv). \]
For a vertex $u$ in $G$ and a subgraph $H$ of $G$, $s_H(u)$ represents the vertex sum of $u$ in $H$. Similarly, for a vertex $u$ in $\vec{G}$ and a subgraph $\vec{H}$ of $\vec{G}$, $\vec{s}_{\vec{H}}(u)$ denotes the oriented vertex sum of $u$ in $\vec{H}$. Given a caterpillar $C$, Theorem~\ref{th:caterpillar} constructs an orientation $\vec{C}$ of $C$ and an antimagic labeling for $\vec{C}$. In the proof, antimagicness will be the consequence of obtaining different absolute values for all oriented vertex sums. For convenience, then, we define the {\em weight} of a vertex $u$ in an oriented graph $\vec{G}$ as $w(u) = |\vec{s}(u)|$. For a subgraph $\vec{H}$ of $\vec{G}$, $w_{\vec{H}}(u) =  |\vec{s}_{\vec{H}}(u)|$ refers to the weight of $u$ in $\vec{H}$.

We also use the notation $[a,b] = \{a,a+1,\dots,b\}$ for any two integers $a, b$ such that $a \le b$.

\goodbreak

\begin{theorem}\label{th:caterpillar}
Every caterpillar admits an antimagic orientation. 
\end{theorem}
\begin{proof}
Let $C$ be a caterpillar with $m$ edges and $r$ leaves, and let $(u_0,\dots,u_{m-r+2})$ be a longest path in $C$. Then, we define the path $P = (u_0,\dots,u_k)$, where $k = m-r+2$ if $m-r$ is even and $k = m-r+1$ otherwise. The number of edges of $P$ is, therefore, even in any case. From here on we will refer to the edges of $P$ as {\em path edges}, and to the rest of edges in $C$ as {\em non-path} edges. 

We describe now an algorithm in seven steps to construct an orientation $\vec{C}$ of $C$ and a labeling $\phi: A(\vec{C}) \rightarrow [1,m]$ that will then be shown to be antimagic. %At any particular stage of the algorithm, $s(u)$, $\vec{s}(u)$, and $w(u)$ will respectively denote the vertex sum, the directed vertex sum, and the weight of $u$ where only the labels that have been already assigned up to that point are used. 

\begin{enumerate}
\item {\em Defining the label set.}
Consider the division of the label set $L = [1,m]$ into the three subsets $L_1 = [1,k_1]$, $L_2 = [k_1+1,k_2]$, and $L_3 = [k_2+1,m]$,
where 
\[k_1 = \Bigl\lceil \frac{m-r+1}{2} \Bigr\rceil, \;\mbox{ and }\; k_2 = \Bigl\lceil \frac{m+r}{2} \Bigr\rceil - 1.\] 
Labels in $L_1$ and $L_3$ will be used for the path edges, and labels in $L_2$ for the non-path edges. The equality $k_1 + k_2 = m$ (see Claim~\ref{cl:m}) will be applied throughout the proof.

\item {\em Labeling the path edges.}
We assign the labels in $L_1$ and $L_3$ to the path edges in an alternating way. The largest label in $L_1$, $k_1$, is assigned to the edge $\{u_0,u_1\}$, then the largest label in $L_3$, $m$, is assigned to the next edge, $\{u_1,u_2\}$, then the previous labels are assigned to the next edges in $P$ and so on, keeping the alternation of labels until the first ones are reached. The sequence of labels assigned to the path edges is depicted in Figure~\ref{fig:path}. Formally, for $1 \le i < k$, the labeling is defined as:

\begin{equation*}
 \phi(\{u_i,u_{i+1}\}) =
    \begin{cases*}
      k_1-\frac{i}{2}, & \mbox{if $i$ is even} \\
      m - \frac{i-1}{2}, & otherwise
    \end{cases*}
\end{equation*}

\begin{figure}[h]
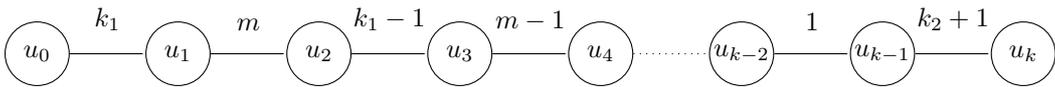

\centerline{\hglue 26pt \path}
\caption{Labeling of the path $P$.}
\label{fig:path}
\end{figure}

\item {\em Classifying the path vertices and the non-path edges.}
Let $u_i$ be a path vertex with degree at least two in $C$. Then, we call $u_i$ {\em light} if $s_P(u_i) < m$ and $u_i$ is adjacent to exactly one vertex not belonging to $P$; otherwise, $u_i$ is called {\em heavy}. Note that the path vertices of degree one, in particular $u_0$, and also $u_k$ when $m-r$ is even, have degree one in $C$ and, then, are neither light nor heavy.

A non-path edge is called {\em light} if it is incident with a light vertex, and {\em heavy} if it is incident with a heavy vertex. 

\item {\em Orienting the path edges.} The edge $\{u_0,u_1\}$ is oriented as the arc $u_0 u_1$. The edge $\{u_i,u_{i+1}\}$, for $0 < i < k$, is oriented in the ``same direction" in the path as the edge $\{u_{i-1}, u_i \}$  if $u_i$ is light and in ``contrary direction" if $u_i$ is heavy, that is, as:
\begin{enumerate}
\item $u_i u_{i+1}$ if either $u_{i-1} u_i \in A(\vec{C})$ and $u_i$ is light or $u_i u_{i-1} \in A(\vec{C})$ and $u_i$ is heavy,
\item $u_{i+1} u_i$ if either $u_{i-1} u_i \in A(\vec{C})$ and $u_i$ is heavy or $u_i u_{i-1} \in A(\vec{C})$ and $u_i$ is light.
\end{enumerate}
The oriented path $P$ will be represented with $\vec{P}$.

\item {\em Orienting the non-path edges.} 
Note that any non-path edge is incident with a light or a heavy vertex. We consider two cases for a non-path edge $\{u_i,v\}$ where $u_i$ is a path vertex and $v$ is a leaf:
\begin{enumerate}
\item if $u_i$ is light, we orient $\{u_i,v\}$ as $u_i v$ if 
$\vec{s}_{\vec{P}}(u_i) > 0$, and as $v u_i$ if $\vec{s}_{\vec{P}}(u_i) < 0$,
\item if $u_i$ is heavy, we orient $\{u_i,v\}$ as $u_i v$ if $\vec{s}_{\vec{P}}(u_i) < 0$, and as $v u_i$ if $\vec{s}_{\vec{P}}(u_i) > 0$.
\end{enumerate}
The goal of this orientation is that heavy edges help maximizing the weight of the heavy vertices they are incident with, while light edges help minimizing the weight of their respective light vertices.
 
\item {\em Labeling the light edges.} 
Let $u_{i_1}, u_{i_2}, \dots, u_{i_{n_l}}$ be an ordering of the $n_l$ light vertices such that $w_{\vec{P}}(u_{i_1}) \le w_{\vec{P}}(u_{i_2}) \le \dots \le w_{\vec{P}}(u_{i_{n_l}})$. Let $l_t$ be the light edge incident with $u_{i_t}$, for $1 \le t \le n_l$; then, we define $\phi(l_t) = k_2 - t + 1$.

\item {\em Labeling the heavy edges.}
For each heavy vertex $u_j$ with at least two incident heavy edges, we randomly assign unused labels from $L_2$ to all heavy edges incident with $u_j$ except one. Now, all heavy vertices must be incident with at most one heavy edge which has not yet been assigned a label. Let $u_{j_1},u_{j_2},\dots,u_{j_{n_h}}$ be an ordering of the $n_h$ heavy vertices having one incident still unlabeled heavy edge such that $w'(u_{j_1}) \le w'(u_{j_2}) \le \dots \le w'(u_{j_{n_h}})$, where $w'(u_{j_1}), w'(u_{j_2}), \dots, w'(u_{j_{n_h}})$ are the partial weights calculated with the labels assigned so far. Let $h_t$ be the still unlabeled heavy edge incident with $j_t$, for $1 \le t \le n_h$; then, we define $\phi(h_t)$ as the $t$-th smallest unused label in $L_2$.
\end{enumerate}

Let $\vec{C}$ be the orientation of $C$ defined above. Now, we establish some facts before proving that $\phi$ is an antimagic labeling of $\vec{C}$.

\begin{claim}\label{cl:m}
It holds that $k_1+k_2=m$.
\end{claim}
\begin{proof}
We use well-known transformations of the ceiling and floor functions~\cite{GKP}, as indicated in the side annotations for a real $x$ and an integer $n$:
\begin{align*}
m &= m + r - r \\
&= \Bigl\lfloor \frac{m+r}{2} \Bigr\rfloor +
\Bigl\lceil \frac{m+r}{2}\Bigr\rceil - r && \text{since $n = \floor{n/2}+\ceil{n/2}$} \\
&= \Bigl\lfloor \frac{m-r}{2} \Bigr\rfloor +
\Bigl\lceil \frac{m+r}{2} \Bigr\rceil && \text{since $\floor{x}+n = \floor{x+n}$} \\
&= \Bigl\lceil \frac{m-r-1}{2} \Bigr\rceil +
\Bigl\lceil \frac{m+r}{2} \Bigr\rceil && \text{since $\Big\lfloor\frac{n}{2}\Big\rfloor = \Big\lceil\frac{n-1}{2}\Big\rceil$} \\
&= \Bigl\lceil \frac{m-r+1}{2} \Bigr\rceil +
\Bigl\lceil \frac{m+r}{2} \Bigr\rceil - 1 && \text{since $\ceil{x}+n = \ceil{x+n}$} 
\end{align*}
\end{proof}

\goodbreak

\begin{claim}\label{cl:constant}
For any $t$ such that $1 \le t \le n_l$, it holds that $k_2 \le w_{\vec{P}}(u_{i_t}) \le k_2 + 1$.
\end{claim}
\begin{proof}
Let  $u_{i_t}$ be a light vertex. If $i_t = k =  m-r+1$ (in the case $m-r$ is odd), then we know that $u_{i_t}$ is an endpoint of the path and is incident with a light edge; therefore, $w_{\vec{P}}(u_{i_t}) = k_2+1$ from step 2 in the algorithm. Otherwise, since $u_0$ is not a light vertex, suppose $u_{i_t}$ is not an endpoint of the path. The orientation of the edges done in step 4 implies that the weight of $u_{i_t}$ in the path, $w_{\vec{P}}(u_{i_t})$, is the result of subtracting the labels of the edges incident with $u_{i_t}$ (defined in step 2):
\[ w_{\vec{P}}(u_{i_t}) = |\phi(\{u_{i_t-1},u_{i_t}\}) - \phi(\{u_{i_t},u_{i_t+1}\})|. \]
If $i_t$ is even, then
\[ w_{\vec{P}}(u_{i_t}) = \Big|\Big(m - \frac{(i_t-1)-1}{2}\Big) - \Big(k_1 - \frac{i_t}{2}\Big)\Big| 
= (m- i_t/2 + 1) - (k_1 - i_t/2) 
\]\[
= m - k_1 + 1 = k_2 + 1,\]
where Claim~\ref{cl:m} has been applied at the last equality. Similarly, if $i_t$ is odd, then
\[ w_{\vec{P}}(u_{i_t}) = \Big|\Big(k_1-\frac{i_t-1}{2}\Big) - \Big(m - \frac{i_t-1}{2}\Big) \Big|
= |k_1 - m| = |-k_2| = k_2. \]
Therefore, $k_2 \le w_{\vec{P}}(u_{i_t}) \le k_2 + 1$ for any light vertex $u_{i_t}$.
\end{proof}

\begin{claim}\label{cl:light}
It holds that $n_l \le k_1 - 1$.
\end{claim}
\begin{proof}
For any light vertex $u_i$ it holds that $s_P(u_i) < m$. Since all possible vertex sums in the path that are smaller than $m$ are $k_2+1, k_2+2, \dots, k_2+k_1-1 = m-1$, the maximum number of light vertices $n_l$ must necessarily be at most $k_1-1$.
\end{proof}

Now we prove that $\phi$ is an antimagic labeling of $\vec{C}$. We do this by classifying the vertices into different classes depending on their degree and on the fact of being light or heavy. Then, we show that the weights are pairwise distinct inside each class and that the sets of weights for the classes are pairwise disjoint. As a conclusion, any oriented vertex sums at two different vertices in $\vec{C}$ must be distinct because, otherwise, their absolute values (that is, their weights) would coincide.

\subsubsection*{Light vertices}

Note that, as a consequence of the orientation (step 5) and labeling (step 6) of the light edges, the weight of a light vertex $u_{i_t}$, for $1 \le t \le n_l$, is $w(u_{i_t}) = |w_{\vec{P}}(u_{i_t}) - (k_2-t+1)|$. 

In the first place, we observe that the weights of any two light vertices $u_{i_s}$ and $u_{i_t}$, with $1 \le s < t \le n_l$, must be different:
\begin{align*}
w(u_{i_s}) &= |w_{\vec{P}}(u_{i_s}) - (k_2-s+1)| \\
&= w_{\vec{P}}(u_{i_s}) - (k_2-s+1) && \text{since $w_{\vec{P}}(u_{i_s}) \ge k_2$ by Claim~\ref{cl:constant}}\\ 
&\le w_{\vec{P}}(u_{i_t}) - (k_2-s+1) && \text{since $w_{\vec{P}}(u_{i_s}) \le w_{\vec{P}}(u_{i_t})$ by assumption (step 6)} \\
&= |w_{\vec{P}}(u_{i_t}) - (k_2-s+1)| && \text{since $w_{\vec{P}}(u_{i_t}) \ge k_2$ by Claim~\ref{cl:constant}}\\ 
&< |w_{\vec{P}}(u_{i_t}) - (k_2-t+1)| \\
&= w(u_{i_t}).
\end{align*}
In the second place, we show that the weights of light vertices are distinct from the weights of the rest of vertices in $C$. The smallest weight of a light vertex is that of $u_{i_1}$, $w(u_{i_1}) = |w_{\vec{P}}(u_{i_1}) - k_2|$, which is 0 or 1 by Claim~\ref{cl:constant}. As for the largest weight of a light vertex, we have
\begin{align*}
w(u_{i_{n_l}}) &= |w_{\vec{P}}(u_{i_{n_l}}) - (k_2-n_l+1)| \\
&\le |(k_2+1) - (k_2-n_l+1)| && \text{by Claim~\ref{cl:constant}} \\
&\le |(k_2+1) - (k_2 - (k_1-1) + 1)| && \text{by Claim~\ref{cl:light}} \\
&= k_1 - 1.
\end{align*}
Therefore, the weights of light vertices are pairwise distinct and belong to $[0,k_1-1]$.

\subsubsection*{Vertices of degree one}

Vertices of degree one that belong to the path are at most two: $u_0$, with weight $k_1$, and $u_k$ if $m-r$ is even, in which case $w(u_k) = k_2+1$. Vertices of degree one not belonging to the path, however, have a weight corresponding to the value of a label in $L_2$, hence belonging to the set $[k_1+1,k_2]$. Clearly, then, all weights of vertices of degree one are pairwise distinct and belong to the set $[k_1,k_2+1]$, being larger than the weights of light vertices.

\subsubsection*{Heavy vertices with no incident heavy edge}

Let $u_j$ be a heavy vertex with no incident heavy edge. Since heavy vertices have degree at least two in $C$, $u_j$ must have two path vertices as neighbors and, then, its incident path edges have been oriented in contrary directions in step 4, that is, either as $u_{j-1} u_j$ and $u_{j+1} u_j$ or as $u_j u_{j-1}$ and $u_j u_{j+1}$. Then, the labeling $\phi$ defined in step 2 (see also Figure~\ref{fig:path}) ensures that if $u_j$ and $u_{j'}$ are two heavy vertices with no incident heavy edge, for $0 < j < j' < k$, we have $w(u_j) > w(u_{j'})$,  being $m+k_1$ the largest weight and $k_2+2$ the smallest one. Therefore, the weights of all heavy vertices with no incident heavy edge are pairwise distinct and belong to $[k_2 + 2, m+k_1]$, being larger than the weights of light vertices and of vertices of degree one.

\subsubsection*{Heavy vertices with incident heavy edges}

The orientation of a heavy edge defined in step 5 ensures that any label assigned to it contributes to an increase in the weight of the heavy vertex it is incident with. Consider the list of nondecreasing partial weights $w'(u_{j_1}) \le w'(u_{j_2}) \le \dots \le w'(u_{j_{n_h}})$ from step 7, and consider two partial weights $w'(u_{j_s})$ and $w'(u_{j_t})$ from this list, with $1 \le s < t \le n_h$; remember also that $h_s$ and $h_t$ are the unlabeled heavy edges incident with, respectively, $u_{j_s}$ and $u_{j_t}$. Then, we have that
\[ w(u_{j_s}) = w'(u_{j_s}) + \phi(h_s) < w'(u_{j_s}) + \phi(h_t) \le w'(u_{j_t}) + \phi(h_t) = w(u_{j_t})\]
and, therefore, the weights of any two different heavy vertices $u_{j_s}$ and $u_{j_t}$ are always different. 

Moreover, we can argue that the weight of any heavy vertex $u_j$ with incident heavy edges is $w(u_j) \ge m+k_1+1$. We consider the two possibilities for a heavy vertex $u_j$ according to its definition:
\begin{enumerate}
\item {\em $w_{\vec{P}}(u_j) < m$ and $u_j$ has at least two incident heavy edges.} In this case, $w_{\vec{P}}(u_j) \ge k_2+1$ and the sum of the weights from its incident heavy edges must be at least $2k_1+3$, the sum of the two smallest labels in $L_2$. Then, $w(u_j) \ge (k_2+k_1)+k_1+4 = m+k_1+4$.
\item {\em $w_{\vec{P}}(u_j) \ge m$ and $u_j$ has at least one incident heavy edge.} In this case, the sum of the weights from its incident heavy edges is at least  $k_1+1$, the value of the smallest label in $L_2$. Then, $w(u_j) \ge m+k_1+1$.
\end{enumerate}

Therefore, in any case we have that $w(u_j) \ge m+k_1+1$, which is larger than the weight of any other vertex in the previous classes. 
\end{proof}

%\medskip

\medskip

\begin{figure}[h]
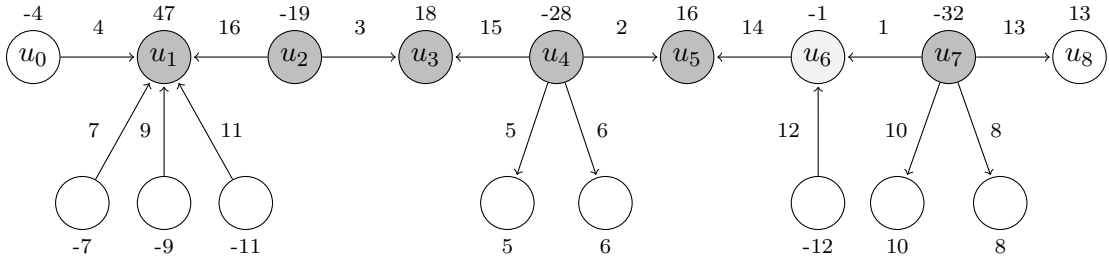

\centerline{\caterpillar1}
\caption{Antimagic orientation of a caterpillar. The only light vertex is $u_6$, colored in light grey; heavy vertices are colored in dark grey.}
\label{fig:example}
\end{figure}

As an example of Theorem~\ref{th:caterpillar}, consider the caterpillar depicted in Figure~\ref{fig:example}. In the notation of the theorem, it has $m=16$ edges and $r=10$ leaves. Since a longest path has 8 edges, which is even, we define $P$ as $(u_0,\dots,u_8)$ without shrinking the longest path by one. In step 1 we set the values $k_1 = 4$, and $k_2 = 12$ and divide the label set $[1,16]$ into $L_1 = [1,4]$, $L_2 = [5,12]$, and $L_3 = [13,16]$. Then, $P$ is labeled alternating the labels from $L_1$ and $L_3$, as indicated in step 2. Now, step 3 identifies vertex $u_6$ as light because $s(u_6) = 15 < m$ (its incident path edges are then oriented in the same direction in step 4), while the rest of path vertices which are not endpoints are defined as heavy vertices (and their incident path edges oriented in contrary directions in step 4). In step 5, heavy edges are oriented in the direction that maximizes the weight of their respective heavy vertices, while the only light edge (incident with $u_6$) is oriented in such a way that the weight of $u_6$ be minimized. This is now done in step 6, where label 12, the largest one in $L_2$, is assigned to the light edge incident with $u_6$. Step 7 starts assigning random labels from $L_2$ to four heavy edges, labels $7$, $9$, $5$, and $10$ in the example, leaving one heavy edge unlabeled for each of the vertices $u_1$, $u_4$, and $u_7$. The way to assign the remaining labels to the still unlabeled heavy edges is the following. First, calculate the list of partial weights of the heavy vertices having incident heavy edges by nondecreasing weight: $w'(u_4) = 22$, $w'(u_7) = 24$, and $w'(u_1) = 36$. Then, assign the remaining labels by increasing value (6, 8, and 11) to the unlabeled heavy edge incident with each vertex in the list. The final weights are strictly increasing in the same ordering and, therefore, pairwise distinct.

\end{document}